\documentclass[12pt, reqno]{amsart}
\usepackage{amsmath, amsthm, amscd, amsfonts, amssymb, graphicx, color}
\usepackage[bookmarksnumbered, colorlinks, plainpages]{hyperref}
\hypersetup{colorlinks=true,linkcolor=red, anchorcolor=green, citecolor=cyan, urlcolor=red, filecolor=magenta, pdftoolbar=true}

\newtheorem{theorem}{Theorem}[section]
\newtheorem{lemma}[theorem]{Lemma}
\newtheorem{proposition}[theorem]{Proposition}
\newtheorem{corollary}[theorem]{Corollary}
\theoremstyle{definition}
\newtheorem{definition}[theorem]{Definition}
\newtheorem{example}[theorem]{Example}

\theoremstyle{remark}
\newtheorem{remark}[theorem]{Remark}
\numberwithin{equation}{section}

\begin{document}

\setcounter{page}{1}

\title[Orthogonality and parallelism of operators on Banach spaces]{Orthogonality and parallelism of operators on various Banach spaces}

\author[T. Bottazzi, C. Conde, M.S. Moslehian, P. W\'ojcik and A. Zamani]{T. Bottazzi$^1$, C. Conde$^{1,2}$, M. S. Moslehian$^3$, P. W\'ojcik$^4$ and A. Zamani$^{3,5}$}

\address{$^1$Instituto Argentino de Matem\'atica ``Alberto P. Calder\'on", Saavedra 15 3U piso, (C1083ACA) Buenos Aires, Argentina}
\email{tpbottaz@ungs.edu.ar}

\address{$^2$Instituto de Ciencias, Universidad Nacional de Gral. Sarmiento, J. M. Gutierrez 1150, (B1613GSX) Los Polvorines, Argentina}
\email{cconde@ungs.edu.ar}

\address{$^3$Department of Pure Mathematics, Ferdowsi University of Mashhad, Center of Excellence in
Analysis on Algebraic Structures (CEAAS), P.O. Box 1159, Mashhad 91775, Iran}
\email{moslehian@um.ac.ir, moslehian@member.ams.org}

\address{$^4$Institute of Mathematics, Pedagogical University of Cracow, Podchor\c a\.zych 2, 30-084 Krak\'ow, Poland}
\email{pwojcik@up.krakow.pl}

\address{$^5$Department of Mathematics, Farhangian University, Iran}
\email{zamani.ali85@yahoo.com}


 \begin{abstract}
We present some properties of orthogonality and relate them with support disjoint and norm inequalities in $p-$Schatten ideals.
In addition, we investigate the problem of characterization of norm--parallelism for bounded linear operators.
We consider the characterization of norm--parallelism problem in $p-$Schatten ideals and locally uniformly convex spaces.
Later on, we study the case when an operator is norm--parallel to the identity operator.
Finally, we give some equivalence assertions about the norm--parallelism of compact operators.
Some applications and generalizations are discussed for certain operators.
\end{abstract}
\maketitle

\section{Introduction}
Let $(\mathcal{X}, \|\cdot\|)$ be a normed space over $\mathbb{K} \in \{\mathbb{R}, \mathbb{C}\}$.
The orthogonality between two vectors of $\mathcal{X}$, may be defined in several ways.
The so-called \textit{Birkhoff--James} orthogonality reads as follows (see \cite{B, J}):
for $x,y \in \mathcal{X}$ it is said that $x$ is \textit{Birkhoff--James orthogonal} (B-J) to $y$, denoted by $x \perp_{BJ} y$, whenever
\begin{equation} \label{defiBJeq}
\|x\|\leq \|x + \gamma y\|
\end{equation}
for all $\gamma \in \mathbb{K}$. If $\mathcal{X}$ is an inner product space, then B-J orthogonality is equivalent
to the usual orthogonality given by the inner product.
It is also easy to see that B-J orthogonality is nondegenerate, is homogeneous, but it is neither symmetric nor additive.

There are other definitions of orthogonality with different properties. We focus in particular the previous notion and isosceles orthogonality, which is defined as follows: in a real normed space $\mathcal{X}$, we say that $x\in \mathcal{X}$ is \textit{isosceles orthogonal} to $y\in \mathcal{X}$ (see \cite{J}) and we write $x \perp_{I} y$ whenever
\begin{equation} \label{defiI}
\|x+y\|= \|x-y\|.
\end{equation}
In complex normed spaces one have to consider the following orthogonality relation
\begin{equation} \label{defiI2}
x\perp_{I}y\Leftrightarrow \left\{
\begin{array}{c l r l}
 \|x+y\|= \|x-y\|\\
 \|x+iy\|= \|x-iy\|. \\
\end{array}
\right.
\end{equation}
For a recent account of the theory of orthogonality in normed linear spaces we refer the reader to \cite{A.M.W}.

Orthogonality in the setting of Hilbert space operators has attracted attention of several mathematicians.
We cite some papers which are closer to our results in chronological order.
Stampfli \cite{stampfli} characterized when an operator is B-J orthogonal to the identity operator.
This result was generalized to any pair of operators by Magajna \cite{magajna}.
In \cite{kittaneh_laa_1991}, Kittaneh gave necessary and sufficient conditions such that $I\perp_{BJ} A$ for $p-$norm in a finite dimensional context.
Also for matrices, Bhatia and {\v{S}}emrl \cite{B.S} obtained a generalization of Kittaneh's result and other statements concerning the spectral norm. More recently, some other authors studied different aspects of orthogonality of bounded linear operators and elements of an arbitrary Hilbert $C^*$-module, for instance, see \cite{benitez-fernandez-soriano, B.G, P.S.G, sain-paul-hait}.

Furthermore, we say that $x\in \mathcal{X}$ is \textit{norm--parallel} to $y\in \mathcal{X}$ (see \cite{S, Z.M.2}), in short $x\parallel y$, if there exists $\lambda\in\mathbb{T}=\{\alpha\in\mathbb{K}: \,\,|\alpha|=1\}$ such that
\begin{equation} \label{defiparallel}
\|x + \lambda y\| = \|x\| + \|y\|.
\end{equation}
In the framework of inner product spaces, the norm--parallel relation is exactly the usual vectorial parallel relation, that is,
$x\parallel y$ if and only if $x$ and $y$ are linearly dependent. In the setting of normed linear spaces, two linearly
dependent vectors are norm--parallel, but the converse is false in general.
To see this consider the vectors $(1, 0)$ and $(1, 1)$ in the space $\mathbb{C}^2$ with the max--norm.
Notice that the norm--parallelism is symmetric and $\mathbb{R}$-homogenous, but not transitive
(i.e., $x\parallel y$ and $y\parallel z \nRightarrow x\parallel z$; see \cite[Example 2.7]{Z.M.2}, unless $\mathcal{X}$ is smooth at $y$; see \cite[Theorem 3.1]{W}).
It was shown in \cite{Z.M.2} that the following relation between the norm--parallelism and B-J orthogonality is valid:
\begin{equation} \label{rel.N-P.BJ}
x\parallel y \Leftrightarrow \big(x\perp_{BJ}(\|y\|x + \lambda \|x\|y)\text{~for some~} \lambda\in\mathbb{T}\big ).
\end{equation}
Some characterizations of the norm--parallelism for Hilbert space operators
and elements of an arbitrary Hilbert $C^*$-module were given in \cite{G, W, Z, Z.M.1, Z.M.2}.

We briefly describe the contents of this paper. Section 2 contains basic definitions, notation and some preliminary results.
In section 3, we present some properties of orthogonality and relate them with support disjoint and norm inequalities in $p-$Schatten ideals.
In the last section, we restrict our attention to the problem of characterization of norm--parallelism for bounded linear operators.
We first consider the characterization of norm--parallelism problem in $p-$Schatten ideals and locally uniformly convex spaces.
Later on we investigate the case when an operator is norm--parallel to the identity operator.
Finally, we give some equivalence assertions about the norm--parallelism of compact operators.
Some applications and generalizations are discussed for certain operators.
\section{Preliminaries}
Throughout the paper, $\mathcal{X}$, $\mathcal{Y}$ stand for normed spaces. Further, $(\mathcal{H}, \langle \cdot, \cdot\rangle)$ denotes a separable complex Hilbert space.
We write $\mathbb{B}_\mathcal{X}$ and $\mathbb{S}_\mathcal{X}$, respectively,
to show the closed unit ball and the unit sphere of $\mathcal{X}$. The (topological) dual of
$\mathcal{X}$ is denoted by $\mathcal{X}^*$. If there is a unique supporting hyperplane at each point of $\mathbb{S}_\mathcal{X}$, then $\mathcal{X}$ is said to be smooth.
A space $\mathcal{X}$ is said to be strictly convex if every element of $\mathbb{S}_\mathcal{X}$ is an extreme
point of $\mathbb{B}_\mathcal{X}$. Further, $\mathcal{X}$ is said to be uniformly convex if for any sequences $\{x_n\}$ and $\{y_n\}$
in $\mathbb{B}_\mathcal{X}$ with $\lim\limits_{n\rightarrow\infty}\|x_n + y_n\| = 2$, we have $\lim\limits_{n\rightarrow\infty}\|x_n - y_n\| = 0$.
The concept of strictly convex and uniformly convex spaces have been extremely
useful in the study of the geometry of Banach Spaces (see \cite{J.2}).

Let $\mathbb{B}(\mathcal{X}, \mathcal{Y})$ and $\mathbb{K}(\mathcal{X}, \mathcal{Y})$ denote the Banach spaces of all bounded operators and all compact operators equipped with the operator norm, respectively.
We write $\mathbb{K}(\mathcal{X}, \mathcal{Y}) = \mathbb{K}(\mathcal{X})$ and $\mathbb{B}(\mathcal{X}, \mathcal{Y}) = \mathbb{B}(\mathcal{X})$ if $\mathcal{X} =\mathcal{Y}$. The symbol $I$ stands for the identity operator on $\mathcal{X}$. In addition, we denote by $\mathbb{M}_A$ the set of all unit vectors at which $A$ attains its norm, i.e., $\mathbb{M}_A = \{x\in \mathbb{S}_\mathcal{X};\,\, \|Ax\| = \|A\|\}.$

For $A \in \mathbb{B}(\mathcal{H})$ we use $R(A)$ and $\ker(A)$, respectively, to
denote the range and kernel of $A$.
We say that $A \leq B$ whenever $\langle Ax, x\rangle \leq \langle Bx, x\rangle$
for all $x\in \mathcal{H}$. An element $A\in \mathbb{B}(\mathcal{H})$ with $A\geq 0$ is called positive. For any $\mathcal{L}\subseteq \mathbb{B}(\mathcal{H})$, ${\mathcal{L}}^+$ denotes the subset of all positive operators of $\mathcal{L}$.

For any compact operator $A\in \mathbb{K}(\mathcal{H})$, let $s_1(A), s_2(A),\cdots $ be the singular values of $A$, i.e.
the eigenvalues of the ``absolute value-norm" $|A| = (A^*A)^{\frac {1}{2}}$ of $A$, in decreasing order and repeated
according to multiplicity. Here $A^*$ denotes the adjoint of $A$. If $A\in \mathbb{K}(\mathcal{H})$ and $p>0$, let
\begin{equation} \label{defipllel}
{\|A\|}_p = \left(\sum_{i = 1}^\infty s_i(A)^p\right)^{\frac{1}{p}} = \left({\rm tr}|A|^p\right)^{\frac{1}{p}},
\end{equation}
where $\rm tr$ is the usual trace functional, i.e.
${\rm tr}(A)=\sum_{j=1}^{\infty} \langle Ae_j,e_j\rangle,$
where $\{e_j\}_{j=1}^{\infty}$ is an orthonormal basis of $\mathcal{H}$. Equality \eqref{defipllel}
defines a norm (quasi-norm) on the ideal $\mathbb{B}_p(\mathcal{H}) = \{A\in \mathbb{K}(\mathcal{H}): {\|A\|}_p<\infty\}$ for $1\leq p<\infty$ ($0< p <1$), called the
\textit{$p$-Schatten class}.
It is known that the so-called \textit{Hilbert-Schmidt class} $\mathbb{B}_2(\mathcal{H})$ is a Hilbert space under the inner product $\langle A, B \rangle_{HS} := {\rm tr}(B^*A)$.
The ideal $\mathbb{B}_1(\mathcal{H})$ is called the \textit{trace class}. It is not reflexive and, in particular, is not a uniformly convex space, because it
contains a subspace isomorphic to $l_1$ (the subspace can be chosen to be the operators diagonal with respect to a
given orthonormal basis of $\mathcal{H}$).

According to \cite{arazy}, we define the concept of disjoint supports as follows.
\begin{definition} Let $A \in \mathbb{B}(\mathcal{H})$. The right support $r(A)$ of $A$ is the orthogonal projection of $\mathcal{H}$ onto
$\ker(A)^{\perp}=\{h\in \mathcal{H}:\, \langle h, x \rangle=0 \text{~for all~} x\in \ker(A)\}$ and the left support $ l(A)$ of $A$ is the orthogonal projection of
$\mathcal{H}$ onto $\overline{R(A)}$.
Two operators $A, B \in \mathbb{B}(\mathcal{H})$ have:
\begin{itemize}
\item[(1)] \textit{right disjoint supports} if and only if $r(A)r(B)=0$.
\item[(2)] \textit{left disjoint supports} if and only if $l(A)l(B)=0$.
\item[(3)] \textit{disjoint supports} if and only if $r(A)r(B)=0$ and $l(A)l(B)=0$.
\end{itemize}
\end{definition}
Let us recall that if $P_{N}$ denotes the orthogonal projection onto the closed subspace $N$ of $\mathcal{H}$, then $P_{N}P_{Q}=0$
if and only if $N\perp Q$ (i.e. $\langle x, y\rangle=0$ for all $x\in N, y\in Q$).
We observe that $r(A)r(B)=0$ if and only if $AB^*=0$ and similarly with the left disjoint support.
Consequently, two operators $A, B$ have
disjoint supports if and only if $R(A)\perp R(B)$ and $R(A^*)\perp R(B^*)$.

The following is a well-known result, which we use in the present article.
Let $A, B \in \mathbb{B}_p(\mathcal{H})$. Then, the following conditions are equivalent:
\begin{itemize}
\item[(1)] $A^*AB^*B = 0$ (or $AA^*BB^*=0$).
\item[(2)] $|A||B|=|B||A|=|A^*||B^*|=|B^*||A^*|=0$.
\item[(3)] $A$ and $B$ have disjoint supports.
\end{itemize}
The classical Clarkson--McCarthy inequalities assert that if
$A,B\in\mathbb{B}_p(\mathcal{H})$, then
\begin{equation}\label{Clarkson1}
{\|A + B\|}_p^p + {\|A - B\|}_p^p \leq 2({\|A\|}_p^p + {\|B\|}_p^p)
\end{equation}
for $1\leq p\leq2$, and
\begin{equation} \label{Clarkson2}
{\|A + B\|}_p^p + {\|A - B\|}_p^p \geq 2({\|A\|}_p^p + {\|B\|}_p^p)
\end{equation}
for $2\leq p<\infty$. If $p=2$, equality always holds and if $p\neq 2$, equality holds if and only if $A^*AB^*B =0$. This inequality implies the uniform convexity of $\mathbb{B}_p(\mathcal{H})$ for $1 < p < \infty$. If $0<p\leq 1$, then
\begin{equation} \label{Clarkson3}
{\|A+B\|}_p^p\leq {\|A\|}_p^p + {\|B\|}_p^p.
\end{equation}
Furthermore, if $p<1$ the previous equality holds if and only if $A^*AB^*B=0$.
The previous inequality motivates us to study conditions such that equality in \eqref{Clarkson3} holds for $1\leq p<\infty$.
These and more results related with these ideals can be seen, for instance, in \cite{mccarthy} and \cite{gohberg-krein}.
\section{Orthogonality in $\mathbb{B}_p(\mathcal{H})$}
\subsection{Birkhoff--James and isosceles orthogonality in $\mathbb{B}_p(\mathcal{H})$ ideals}
$\quad$

Let $\mathbb{B}_p(\mathcal{H})$ be a $p$-Schatten ideal with $p>0$. Using \eqref{defiBJeq} the Birkhoff--James orthogonality for any $A,B\in \mathbb{B}_p(\mathcal{H})$ is
$$A\perp_{BJ}^p B \Leftrightarrow \big({\|A\|}_p\leq {\| A + \gamma B\|}_p \text{~for all~} \gamma\in \mathbb{C}\big).$$
The following result is a direct consequence of \cite[Lemma 3.1]{maher}, which relates the concept of disjoint supports with B-J orthogonality in $p$-Schatten ideal. It is true for any symmetric ideal associated to a unitarily invariant norm, but in this paper we restrict ourselves to the $p$-Schatten ideals to unify our study.
\begin{proposition} \label{soporte implica bj}
If $A+B\in\mathbb{B}_p(\mathcal{H})$ with $1 \leq p<\infty$ and $R(A)\perp R(B)$, then $A,B\in \mathbb{B}_p(\mathcal{H})$ and
$A\perp_{BJ}^p B$ and $B\perp_{BJ}^p A$.
\end{proposition}
The following statement is a reverse of Proposition \ref{soporte implica bj} for positive operators in $\mathbb{B}_p(\mathcal{H})$.
\begin{theorem} \label{bj implica soporte}
Let $1<p<\infty$ and $A,B\in{\mathbb{B}_p(\mathcal{H})}^+$. If $A\perp_{BJ}^p B$, then $A$ and $B$ have disjoint supports.
\end{theorem}
\begin{proof}
Define the real valued function
$$f(\gamma)=\|A + \gamma B\|_p.$$
Without loss of generality, we may assume that $A\neq 0$. By the hypothesis, $f$ attains its minimum at $\gamma = 0$ and so, by \cite[Theorem 2.3]{abatzoglou},
$$0=\left. \frac{d}{d\gamma}\|A + \gamma B\|_p\right|_{\gamma = 0} = {\rm tr}\left(\frac{A^{p-1}B}{\|A\|_p^{p-1}} \right),$$
whence ${\rm tr}\left(A^{p-1}B \right)=0$. We observe that
\begin{align*}
0 & = {\rm tr}\left(A^{p-1}B \right) = {\rm tr}\left(B^{1/2}\left( A^{p-1}\right)^{1/2}\left( A^{p-1}\right)^{1/2}B^{1/2} \right)\\
& = {\rm tr}\left[\left(( A^{p-1})^{1/2}B^{1/2}\right)^*\left(( A^{p-1})^{1/2}B^{1/2}\right)\right]\\
& =\left\langle \left( A^{p-1}\right)^{1/2}B^{1/2},\left( A^{p-1}\right)^{1/2}B^{1/2} \right\rangle_{HS},
\end{align*}
which implies that $\left( A^{p-1}\right)^{1/2}B^{1/2}=0$. Thus,
$$A^{p-1}B=\left( A^{p-1}\right)^{1/2}\left( A^{p-1}\right)^{1/2}B^{1/2}B^{1/2}=0.$$
Since $A, B$ are compact operators, there exists an orthonormal basis $\{e_i\}_{i\in \mathcal{I}}$ of $\mathcal{H}$ such that
$$A^{p-1} = \sum_{i\in \mathcal{I}}\alpha_i e_i\otimes e_i\quad {\rm and}\quad B = \sum_{i\in \mathcal{I}}\beta_i e_i\otimes e_i,$$
with $\alpha_i,\beta_i\geq 0$ for every $i\in \mathcal{I}$. Hence and $\alpha_i\beta_i=0$ for every $i\in \mathcal{I}$. Finally, we have
$${\rm tr}(AB) = {\rm tr}\left[ \left( \sum_{i\in \mathcal{I}}\alpha_i e_i\otimes e_i\right)^{\frac{1}{p-1}}\left( \sum_{i\in \mathcal{I}}\beta_i e_i\otimes e_i\right)\right] = \sum_{i\in \mathcal{I}}\alpha_i^{\frac{1}{p-1} }\beta_i = 0$$
and this implies that
$$0={\rm tr}(AB) ={\rm tr}\left( \left( A^{1/2}B^{1/2}\right)^* A^{1/2}B^{1/2} \right)=\|A^{1/2}B^{1/2}\|_2^2.$$
\end{proof}
\begin{remark}\label{r.e.m}
The previous result does not hold if $A$ or $B$ are not positive operators. For instance, let $A = \begin{bmatrix}
1/2&0\\
0&-1/2
\end{bmatrix}$ and $I = \begin{bmatrix}
1 & 0 \\
0 & 1
\end{bmatrix}$. Then by Corollary \ref{cr.0248}, since ${\rm tr}(A) = 0$ we get $I\perp_{BJ}^p A$.
Nevertheless, $I$ and $A$ clearly do not have disjoint supports.\\
Furthermore, the previous result does not hold either for $p=1$ or $p=\infty$. Let $B = \begin{bmatrix}
1&0\\
0&0
\end{bmatrix}$. Then it is easy to see that $B\perp_{BJ}^{1} I$ and $I\perp_{BJ}^{\infty} A$.
Nevertheless, $I$ and $B$ clearly do not have disjoint supports.
\end{remark}
In \cite{busch}, it was proved that for any $A, B \in {\mathbb{B}_1(\mathcal{H})}^+$ it holds that
\begin{eqnarray}\label{msm1}
A\perp_{I}^1 B \Leftrightarrow AB = BA =0\,.
\end{eqnarray}
From the positivity of $A, B$ and $A + B$ we infer that
$${\|A + B\|}_1 = {\rm tr}(A + B) = {\rm tr}(A) + {\rm tr}(B) = {\|A\|}_1 + {\|B\|}_1.$$
Recently, Li and Li \cite{li-li} gave a characterization of disjoint supports for operators in the trace class ideal. In the same direction, we present the following result.
\begin{theorem} \label{equivalencia soporte normas}
Let $0<p<\infty $ and $A,B\in \mathbb{B}_p(\mathcal{H})$. Then, the following statements are equivalent:
\begin{itemize}
\item[(i)] $A$ and $B$ have disjoint supports.
\item[(ii)] ${\|\lambda A - \mu B\|}_p^p = {\|\lambda A +\mu B\|}_p^p = |\lambda|^p{\|A\|}_p^p +|\mu|^p {\|B\|}_p^p$, for any $\lambda,\mu\in \mathbb{C}$.
\item[(iii)] ${\|A - B\|}_p^p = {\|A + B\|}_p^p = {\|A\|}_p^p + {\|B\|}_p^p$.
\end{itemize}
\end{theorem}
\begin{proof}
(i)$\Rightarrow$(ii) By the hypothesis $\lambda A$ and $\mu B$ have disjoint supports and $\lambda A + \mu B\in\mathbb{B}_p(\mathcal{H})$, then it follows from
\cite[Theorem 1.7]{maher2} that
$${\|\lambda A +\mu B\|}_p^p = |\lambda|^p{\|A\|}_p^p +|\mu|^p {\|B\|}_p^p.$$
Analogously, ${\|\lambda A - \mu B\|}_p^p =|\lambda|^p{\|A\|}_p^p +|\mu|^p {\|B\|}_p^p$.

(ii)$\Rightarrow$(iii) It is trivial.

(iii)$\Rightarrow$(i) We have
$${\|A + B\|}_p^p + {\|A - B\|}_p^p = 2\left({\|A\|}_p^p + {\|B\|}_p^p\right)$$
or equivalently, by utilizing \cite[Theorem 2.7]{mccarthy}, $A^*AB^*B = 0$, which ensures that $A$ and $B$ have disjoint supports.
\end{proof}
According to \eqref{defiI2}, we say that $A,B\in \mathbb{B}_p(\mathcal{H})$
are isosceles orthogonal, denoted by $A\perp_{I}^p B$ if and only if
$${\|A + B\|}_p = {\|A - B\|}_p\  {\rm and}\ {\|A + i B\|}_p = {\|A - i B\|}_p. $$
The following lemma is a well-known result from McCarthy \cite{mccarthy} for positive operators. We will use it to prove Proposition \ref{paralelogramo isosceles en bp}.
\begin{lemma}[McCarthy inequality] \label{mccarthylema}
If $A, B\in \mathbb{B}_p(\mathcal{H})^+$ for any $p\geq 1$, then
$$2^{1-p}{\|A + B\|}_p^p\leq {\|A\|}_p^p + {\|B\|}_p^p\leq {\|A + B\|}_p^p.$$
\end{lemma}
Observe that in the cone of positive operators $\mathbb{B}_p(\mathcal{H})^+$ we consider the isosceles orthogonality notion as in \eqref{defiI}.
\begin{proposition} \label{paralelogramo isosceles en bp}
    If $A, B\in\mathbb{B}_p(\mathcal{H})^+$, $1\leq p<2$ and $A\perp_{I}^p B$, then
    $${\|A + B\|}_p^p = {\|A - B\|}_p^p = {\|A\|}_p^p + {\|B\|}_p^p.$$
\end{proposition}
\begin{proof}
    By Clarkson--McCarthy inequality \eqref{Clarkson1} and isosceles orthogonality we have
    $$2{\|A + B\|}_p^p = {\|A + B\|}_p^p + {\|A - B\|}_p^p\leq 2\big({\|A\|}_p^p + {\|B\|}_p^p\big)$$
    and
    $${\|A + B\|}_p^p\leq {\|A\|}_p^p + {\|B\|}_p^p.$$
    Now using the previous lemma we obtain that
    $${\|A + B\|}_p^p = {\|A\|}_p^p + {\|B\|}_p^p.$$
\end{proof}

There exists a comparison between isosceles and Birkhoff--James orthogonalities. James in \cite{J} proved that if $\|x - y\| = \|x + y\|$ for any elements $x, y$ in a normed space $\mathcal{X}$, then $\|x\|\leq \|x + \gamma y\|$ for all $|\gamma|\geq 1$. The next statement shows that in the cone of positive operators, both notions of orthogonality coincide.
\begin{theorem}
If $A, B\in\mathbb{B}_p(\mathcal{H})^+$, with $1<p\leq 2$, then the following statements are equivalent:
\begin{itemize}
\item[(i)] $A\perp_{I}^p B$.
\item[(ii)] $A\perp_{BJ}^p B$.
\end{itemize}
Also for $p=1$, the following statements are equivalent:
\begin{itemize}
\item[(iii)] $A\perp_{I}^1 B$.
\item[(iv)] $AB = BA = 0$.
\end{itemize}
\end{theorem}
\begin{proof}
Combining Propositions \ref{soporte implica bj}, \ref{paralelogramo isosceles en bp}, Theorems \ref{bj implica soporte} and \ref{equivalencia soporte normas} we obtain the equivalence desired. On the other hand, the result for $p=1$ is true due to \eqref{msm1}.

We remark that for $p=1$ the equivalence between Birkhoff--James and isosceles orthogonality does not hold. For instance, let $B$ be the same matrix as in Remark \ref{r.e.m}. Then $B\perp_{BJ}^{1} I$ and $\|B+I\|_1\neq \|B-I\|_1$.

\end{proof}
\subsection{$\mathbb{B}_p(\mathcal{H})$ ideals as semi-inner product spaces}
$\quad$

Recall that in any normed space $(\mathcal{X}, \|\cdot\|)$ one can construct (as
noticed by Lumer \cite{lumer} and Giles \cite{giles}) a semi-inner product, i.e., a mapping $[\cdot, \cdot]:\mathcal{X}\times \mathcal{X}\rightarrow\mathbb{K}$ such that
\begin{itemize}
\item[(1)] $[x, x] = \|x\|^2$,
\item[(2)] $[\alpha x + \beta y, z] = \alpha[x, z] + \beta[y, z]$,
\item[(3)] $[x, \gamma y] = \overline{\gamma}[x, y]$,
\item[(4)] $|[x, y]|^2\leq \|x\|^2\|y\|^2$,
\end{itemize}
for all $x, y, z\in \mathcal{X}$ and all $\alpha, \beta, \gamma \in\mathbb{K}$. There may be more than
one semi-inner product on a space. It is well known that in a normed space there exists exactly one semi-inner product if and only if the space is smooth.
If $\mathcal{X}$ is an inner product space, the only semi inner product on $\mathcal{X}$ is the inner product itself.
More details can be found in \cite{lumer, giles}.

Using Lumer's ideas, we endowed the $\mathbb{B}_p(\mathcal{H})$ ideals with semi-inner products as follows.\\
Let $1\leq p<\infty$, we define for any $A, B\in \mathbb{B}_p(\mathcal{H})$
$$[B, A] = {\|A\|}_p^{2-p}{\rm tr}\left(|A|^{p-1}U^*B \right),$$
where $U|A|$ is the polar decomposition of $A$. Then, $[\cdot, \cdot]$ satisfies $(1)$ to $(4)$.
In fact, items $(1)$ to $(3)$ are easily to check. To prove Property $(4)$, we observe that
\begin{align*}
\left|{\rm tr}(|A|^{p-1}U^*B)\right|&\leq {\rm tr}\left||A|^{p-1}U^*B\right| = {\big\||A|^{p-1}U^*B\big\|}_1
\\&\leq {\|B\|}_p{\big\||A|^{p-1}U^*\big\|}_q\leq {\|B\|}_p{\big\||A|^{p-1}\big\|}_q,
\end{align*}
where $\frac{1}{p}+\frac{1}{q} = 1$ and also $${\big\||A|^{p-1}\big\|}_q^q = {\rm tr}\left(|A|^{p-1}\right)^q = {\rm tr}|A|^p = \|A\|_p^p,$$
whence
$${\big\||A|^{p-1}\big\|}_q = {\|A\|}_p^{p/q} = {\|A\|}_p^{p-1}.$$
Thus,
\begin{align*}
\left|[B,A]\right|^2 = \left|{\|A\|}_p^{2-p}\left[{\rm tr}\left(|A|^{p-1}U^*B \right) \right]\right|^2
\leq {\|A\|}_p^{2(2-p)}{\|B\|}_p^2{\|A\|}_p^{2(p-1)} = {\|B\|}_p^2\|A\|_p^2.
\end{align*}
Therefore, $\big(\mathbb{B}_p(\mathcal{H}), [\cdot,\cdot]\big)$ is a semi-inner product space in the sense of Lumer. Moreover,
the continuous property for semi-inner product spaces holds for almost all of these operator ideals, as we state in the following result.
\begin{theorem} \label{teo ortog y trans}
Let $1<p<\infty$. Then, $\big(\mathbb{B}_p(\mathcal{H}), [\cdot,\cdot]\big)$ is a continuous semi-inner product space and for $A, B\in \mathbb{B}_p(\mathcal{H})$
the following statements are equivalent:
\begin{itemize}
\item[(i)] $A\perp_{BJ}^p B$.
\item[(ii)] $[B, A]= 0$.
\end{itemize}
\end{theorem}
\begin{proof}
It is a direct consequence of \cite[Theorem 3]{giles}, since any $p$-norm in $\mathbb{B}_p(\mathcal{H})$ is G\^{a}teaux differentiable for $1<p<\infty$. The characterization of Birkhoff--James orthogonality follows directly from \cite[Theorem 2]{giles}.
\end{proof}
\begin{remark}
Theorem \ref{teo ortog y trans} was been proved for matrices in \cite[Theorem 2.1]{B.S}. In this context, the converse still holds for $p=1$ and $A$ invertible.
\end{remark}
The following result was obtained by Kittaneh in \cite{kittaneh_laa_1991}.
\begin{corollary}\label{cr.0248}
Let $A$ be in the algebra $\mathbb{M}_n(\mathbb{C})$ of all complex $n\times n$ matrices and $p\in[1,\infty)$. Then the following statements are equivalent:
\begin{itemize}
\item[(i)] $I\perp_{BJ}^p A$.
\item[(ii)] ${\rm tr}(A) = 0$.
\end{itemize}
\end{corollary}
\begin{proof}
Theorem \ref{teo ortog y trans} ensures that
$$I\perp_{BJ}^p A\Leftrightarrow [A, I] = 0 =\|I\|_p^{2-p}{\rm tr}\left( |I|^{p-1}U^*A\right)\Leftrightarrow {\rm tr}(A)=0,$$
since $|I|^{p-1}U^*=I$.

On the other hand, we observe that ${\|\cdot\|}_1$ is G\^{a}teaux differentiable at the identity.
\end{proof}
\begin{theorem}\label{th.0004}
Let $A\in \mathbb{B}(\mathcal{H})$. Then, the following conditions are equivalent:
\begin{itemize}
\item[(i)] $A = 0$.
\item[(ii)] $|I + \gamma A| \geq I$ for all $\gamma\in\mathbb{C}$.
\item[(iii)] $|I + \gamma A| = |I - \gamma A|$ for all $\gamma\in\mathbb{C}$.
\end{itemize}
In the case when $\mathcal{H}$ is finite dimensional, for every $p\in[1,\infty)$, each one of these assertions implies that $I\perp_{BJ}^p A$.
\end{theorem}
\begin{proof}
(i) $\Rightarrow$ (ii) is trivial.

(ii) $\Rightarrow$ (iii) Suppose (ii) holds. Then $|I + \gamma A|^2 \geq I$ for all $\gamma\in\mathbb{C}$. Hence
$$\gamma A + \overline{\gamma}A^* + |\gamma|^2|A|^2 \geq 0 \qquad (\gamma\in\mathbb{C}).$$
For $\gamma = \frac{1}{m}, -\frac{1}{m}, \frac{i}{m}, -\frac{i}{m}$, the above inequality becomes
\begin{align} \label{002}
A + A^* + \frac{1}{m}|A|^2 \geq 0, \qquad \qquad A + A^* -\frac{1}{m}|A|^2\leq 0
\end{align}
and
\begin{align} \label{003}
iA -i A^* + \frac{1}{m}|A|^2 \geq 0, \qquad \qquad iA -i A^* -\frac{1}{m}|A|^2\leq 0.
\end{align}
Letting $m\rightarrow \infty$ in (\ref{002}) and (\ref{003}) we get
$$A + A^* = 0 \qquad \mbox{and} \qquad iA -i A^* =0$$
which imply that $A = 0$. Thus $|I + \gamma A| = I = |I - \gamma A|$ for all $\gamma\in\mathbb{C}$.

(iii) $\Rightarrow$ (i) Let $|I + \gamma A| = |I - \gamma A|$ for all $\gamma\in\mathbb{C}$.
Then $\gamma A + \overline{\gamma}A^* = 0$ for all $\gamma\in\mathbb{C}$. For $\gamma = 1, i$, we conclude that
$$A + A^* = 0 \qquad \mbox{and} \qquad iA -i A^* =0$$
which yield that $A = 0$.
\end{proof}

\begin{theorem}\label{th.0005}
Let $A, B\in \mathbb{M}_n(\mathbb{C})$ and $p\in[1,\infty)$. Let $|B + \gamma A| \geq |B|$ for all $\gamma\in\mathbb{C}$.
Then, the following statements hold.
\begin{itemize}
\item[(i)] ${\rm tr}(B^*A) = 0$.
\item[(ii)] $\ker(B + \gamma A) = \ker(B) \cap \ker(A)$ for all $\gamma\in\mathbb{C}\setminus \{0\}$.
\item[(iii)] Either $A$ or $B$ is noninvertible.
\item[(iv)]$B\perp_{BJ}^p A$.
\end{itemize}
\end{theorem}
\begin{proof}
(i) For $\gamma = \frac{1}{m}$, we have $|B + \frac{1}{m} A| \geq |B|$. Hence $s_i(|B + \frac{1}{m} A|) \geq s_i(|B|) \,(1\leq i\leq n)$.
Therefore, $s_i(|B + \frac{1}{m} A|^2) \geq s_i(|B|^2)$ for all $1\leq i\leq n$. We have
\begin{align*}
{\rm tr}(|B|^2) & = \sum_{i=1}^{n}s_i(|B|^2)
\\& \leq \sum_{i=1}^{n}s_i(|B + \frac{1}{m} A|^2)
\\& = {\rm tr}(|B + \frac{1}{m} A|^2)
\\& = {\rm tr}(|B|^2) + \frac{1}{m}{\rm tr}(B^*A) + \frac{1}{m}{\rm tr}(A^*B) + \frac{1}{m^2}{\rm tr}(|A|^2).
\end{align*}
Hence
\begin{align} \label{00201}
{\rm tr}(B^*A) + {\rm tr}(A^*B) + \frac{1}{m}{\rm tr}(|A|^2) \geq 0.
\end{align}
Similarity, for $\gamma = -\frac{1}{m}, \frac{i}{m}, -\frac{i}{m}$, we get
\begin{align} \label{00202}
{\rm tr}(B^*A) + {\rm tr}(A^*B) - \frac{1}{m}{\rm tr}(|A|^2) \leq 0,
\end{align}
\begin{align}\label{00203}
i{\rm tr}(B^*A) - i{\rm tr}(A^*B) + \frac{1}{m}{\rm tr}(|A|^2) \geq 0
\end{align}
and
\begin{align}\label{00204}
i{\rm tr}(B^*A) -i {\rm tr}(A^*B) - \frac{1}{m}{\rm tr}(|A|^2) \leq 0.
\end{align}
Taking $m\rightarrow \infty$ in (\ref{00201}) and (\ref{00202}) we obtain
\begin{align}\label{00205}
{\rm tr}(B^*A) + {\rm tr}(A^*B) = 0.
\end{align}
Also, letting $m\rightarrow \infty$ in (\ref{00203}) and (\ref{00204}) we get
\begin{align}\label{00206}
i{\rm tr}(B^*A) -i {\rm tr}(A^*B) =0.
\end{align}
Now, by (\ref{00205}) and (\ref{00206}), we conclude that ${\rm tr}(B^*A) = 0$.

(ii) Obviously, $\ker(B) \cap \ker(A) \subseteq \ker(B + \gamma A)$. We prove $\ker(B + \gamma A) \subseteq \ker(B) \cap \ker(A)$. Let $x\in \ker(B + \gamma A)$ for $\gamma\neq 0$. First note that if $Bx=0$, then $\gamma Ax=0$, whence $x\in \ker A$. Next, we observe that
$$x\in \ker\left(B + \gamma A\right)=\ker\left(|B + \gamma A| \right).$$
By virtue of the hypothesis we have
$$0=\left\langle |B + \gamma A|x, x\right\rangle \geq\left\langle |B|x,x\right\rangle \geq 0,$$
that is, $x\in \ker(|B|)=\ker(B)$.
We also see that $\ker(B + \gamma A) = \ker(B) \cap \ker(A)$ is equivalent to the range additivity property ${\rm R}(B^* + \overline{\gamma} A^*) = {\rm R}(B^*)+ {\rm R}(A^*)$.

(iii) Now, let $A$ be invertible. Then $\{0\} = \ker(B) \cap \ker(A) = \ker(B + \gamma A)$
for all $\gamma\in\mathbb{C}\setminus \{0\}$. Hence $B + \gamma A$ is invertible for all $\gamma\in\mathbb{C}\setminus \{0\}$. Furthermore, we have $\frac{1}{\gamma}A^{-1}(B + \gamma A) = I + \frac{1}{\gamma}A^{-1}B$ for all $\gamma\in\mathbb{C}\setminus \{0\}$.
Thus the spectrum of $A^{-1}B$ consists of exactly one point. Hence $A^{-1}B$ is noninvertible and so is $B$.

(iv) Let $p\in[1,\infty)$. For an orthonormal basis $\{e_i\}$ consisting of eigenvectors of $B$ and for all $\gamma\in\mathbb{C}$, we have
\begin{align*}
{\|B\|}_p = {\big\||B|\big\|}_p &= \Big(\sum_{i = 1}^{n}{\langle |B|e_i, e_i\rangle}^p\Big)^\frac{1}{p}
\\& \leq \Big(\sum_{i = 1}^{n}{\langle |B + \gamma A|e_i, e_i\rangle}^p\Big)^\frac{1}{p} \leq {\Big\||B + \gamma A|\Big\|}_p = {\|B + \gamma A\|}_p.
\end{align*}
Hence ${\|B\|}_p \leq {\|B + \gamma A\|}_p$, or equivalently $B\perp_{BJ}^p A$.
\end{proof}
The following example shows that statements (i), (ii) and (iii) in the above theorem,
are not equivalent to $|B + \gamma A| \geq |B|$ for all $\gamma\in\mathbb{C}$, in general.
\begin{example}
Let $B =\begin{bmatrix}
1 & 0 \\
0 & 1
\end{bmatrix}$ and $A = \begin{bmatrix}
1 & 1 \\
-1 & -1
\end{bmatrix}$. Simple computations show that ${\rm tr}(B^*A) = 0$, $A$ is noninvertible and $B\perp_{BJ}^p A$ for all $p\in[1,\infty)$.
But for $\gamma = 1$, we have $\frac{1}{\sqrt{2}}\begin{bmatrix}
3 & -1 \\
-1 & 1
\end{bmatrix} = |B + A| \ngeq |B| = \begin{bmatrix}
1 & 0 \\
0 & 1
\end{bmatrix}$.
\end{example}
As a consequence of Theorem \ref{th.0005}, we have the following result.
\begin{corollary}
Let $A, B\in \mathbb{M}_n(\mathbb{C})$ satisfy $B^*A\geq0$. Then, the following conditions are equivalent:
\begin{itemize}
\item[(i)] $B^*A = 0$.
\item[(ii)] $|B + \gamma A| \geq |B|$ for all $\gamma\in\mathbb{C}$.
\end{itemize}
Furthermore, for every $p\in[1,\infty)$, each one of these assertions implies that $B\perp_{BJ}^p A$.
\end{corollary}

\section{Norm--parallelism of operators}
\subsection{Norm--parallelism in $p$-Schatten ideals}
$\quad$

Let $\mathbb{B}_p(\mathcal{H})$ be a $p$-Schatten ideal with $p>0$.
According to \cite{magajna}, we say that $A,B\in \mathbb{B}_p(\mathcal{H})$
are norm--parallel, denoted by $A{\parallel}^p B$,
if there exists $\lambda\in\mathbb{T}$ such that
$${\|A + \lambda B\|}_p = {\|A\|}_p + {\|B\|}_p.$$
The following proposition gives a characterization of parallelism in $\mathbb{B}_p(\mathcal{H})$. This result was previously obtained in \cite{Z.M.1}, for $1<p\leq 2$, with a different proof.
\begin{proposition} \label{paralelo-dependiente}
Let $A, B\in\mathbb{B}_p(\mathcal{H})$ with $1< p<\infty$. Then, the following conditions are equivalent:
\begin{itemize}
\item[(i)] $A{\parallel}^p B$.
\item[(ii)] $A, B$ are linearly dependent.
\end{itemize}
\end{proposition}
\begin{proof}
As we observed, it is evident that if $A$ and $B$ are linearly dependent, then $A{\parallel}^p B$.
Conversely, if $A{\parallel}^p B$, then there exists $\lambda\in\mathbb{T}$ such that
$${\|A + \lambda B\|}_p = {\|A\|}_p + {\|B\|}_p = {\|A\|}_p + {\|\lambda B\|}_p.$$
According to \cite[Corollary 1.5]{ma}, it is equivalent to
$${\left\|\frac{A}{{\|A\|}_p} + \frac{\lambda B}{{\|\lambda B\|}_p}\right\|}_p = 2.$$
Since $\mathbb{B}_p(\mathcal{H})$ is a uniformly convex space for $1<p<\infty$, so ${\left\|\frac{A}{{\|A\|}_p} + \frac{\lambda B}{{\|\lambda B\|}_p}\right\|}_p = 2$ implies that there exists $r\in\mathbb{R}$ such that $B = \frac{r}{\lambda}A$ (see \cite{clk}). Thus $A$ and $B$ are linearly dependent.
\end{proof}
One can similarly show that for any uniformly convex Banach space $\mathcal{X}$ the conditions of parallelism and lineal dependence are equivalent.

The following remark shows that the equivalence between $p-$parallelism and linear dependence does not hold for $p=1, \infty$.
\begin{remark}
Let $A =\begin{bmatrix}
1 & 0 \\
0 & 0
\end{bmatrix}$ and $I = \begin{bmatrix}
1 & 0 \\
0 & 1
\end{bmatrix}$. Then, it is trivial that
$$\|A + I\|_1 = 3 = \|A\|_1 + \|I\|_1,$$
$$\|A + I\| = 2 = \|A\| + \|I\|.$$
It is however evident that $A$ and $I$ are linearly independent.
\end{remark}

In \cite{Z.M.2} the authors related the concept of parallelism to the Birkhoff--James orthogonality for the $p-$Schatten norm and $\mathcal{H}$ a finite dimensional Hilbert space. Utilizing the same ideas, we generalize it to any arbitrary dimension.
\begin{theorem}\label{th.987}
Let $A, B\in\mathbb{B}_p(\mathcal{H})$ with polar decompositions $A=U|A|$ and $B=V|B|$, respectively. If $1<p<\infty$, then the following conditions are equivalent:
\begin{itemize}
\item[(i)] $A{\parallel}^p B$.
\item[(ii)] ${\|A\|}_p\,\big|{\rm tr}(|A|^{p-1}U^*B)\big| = {\|B\|}_p\,{\rm tr}(|A|^p)$.
\item[(iii)] ${\|B\|}_p\,\big|{\rm tr}(|B|^{p-1}V^*A)\big| = \|A\|_p\,{\rm tr}(|B|^p)$.
\item[(iv)] $A, B$ are linearly dependent.
\end{itemize}
\end{theorem}
\begin{remark}\label{re.9876}
In \cite{Z}, for two trace-class operators $A$ and $B$, the following characterizations was proved.
\begin{itemize}
\item[(1)] $A{\parallel}^1 B$.
\item[(2)] There exist a partial isometry $V$ and $\lambda\in\mathbb{T}$ such that $A = V|A|$ and $B = \lambda V|B|.$
\item[(3)] There exists $\lambda\in\mathbb{T}$ such that
 $$\Big|{\rm tr}(|A|) + \mu\, {\rm tr}(U^*B)\Big|\leq\Big{\|P_{\ker A^*}(A + \mu B)P_{\ker A}\Big\|}_1,$$
 where $A = U|A|$ is the polar decomposition of $A$ and $\mu = \frac{{\|A\|}_1}{{\|B\|}_1}\lambda$.
\end{itemize}
If $A$ is invertible, then (1) to (3) are also equivalent to
\begin{itemize}
\item[(4)] $\Big|{\rm tr}\big(|A|A^{-1}B\big)\Big| = {\|B\|}_1$.
\end{itemize}
\end{remark}
As an immediate consequence of Remark \ref{re.9876} (for $p = 1$) and Theorem
\ref{th.987} (for $1<p<\infty$), we have the following result.
\begin{corollary}
Let $A\in \mathbb{M}_n(\mathbb{C})$ and $p\in[1,\infty)$.
Then, the following conditions are equivalent:
\begin{itemize}
\item[(i)] $A{\parallel}^p I$.
\item[(ii)] $|{\rm tr}(A)| = n^{\frac{p-1}{p}}{\|A\|}_p$.
\end{itemize}
\end{corollary}
\subsection{Norm--parallelism to the identity operator}
$\quad$

In this section we investigate the case when an operator is norm--parallel to the identity operator.
In the context of bounded linear operators on Hilbert spaces, the well-known \textit{Daugavet equation}
$$\|A + I\| = \|A\| + 1$$
is a particular case of parallelism; see \cite{WER, Z.M.1} and
the references therein. Such equation is one useful property in solving a variety of problems in approximation theory. We notice that if $A\in \mathbb{B}(\mathcal{H})$ satisfies the Daugavet equation, then $A$ is a normaloid operator, i.e., $\|A\|=w(A)$ where $w(A)$ is the numerical radius of $A$. Reciprocally, a normaloid operator does not necessarily satisfy Daugavet equation, for instance consider $A=-I$.

In order to obtain a characterization of norm--parallelism let us give the following definition.
The numerical radius is the seminorm defined on $\mathbb{B}(\mathcal{X})$ by
$$v(A) : = \sup\Big\{|x^*(Ax)|\,: \,\,x\in \mathbb{S}_\mathcal{X}, x^*\in\mathbb{S}_{\mathcal{X}^*}, x^*(x) = 1\Big\}$$
for each $A\in\mathbb{B}(\mathcal{X})$.
\begin{theorem}\label{th.0}
Let $\mathcal{X}$ be a Banach space and $A\in\mathbb{B}(\mathcal{X})$.
Then the following statements are equivalent:
\begin{itemize}
\item[(i)] $A\parallel I$.
\item[(ii)] $\|A\| = v(A)$.
\end{itemize}
\end{theorem}
\begin{proof}
(i)$\Rightarrow$(ii) Suppose (i) holds. Since the norm--parallelism is symmetric and $\mathbb{R}$-homogenous, so there exists $\lambda\in\mathbb{T}$ such that $$\|I + \lambda r A\| = 1 + |r|\|A\| \qquad (r\in [0,+\infty)).$$
It was shown in \cite{D} that
$$\sup\Big\{\mbox{Re}x^*(\lambda Ax)\,: \,\,x\in \mathbb{S}_\mathcal{X}, x^*\in\mathbb{S}_{\mathcal{X}^*}, x^*(x) = 1\Big\} = \lim_{r\rightarrow 0^{+}} \frac{\|I + \lambda r A\| - 1}{r},$$
hence
$$\sup\Big\{\mbox{Re}x^*(\lambda Ax)\,: \,\,x\in \mathbb{S}_\mathcal{X}, x^*\in\mathbb{S}_{\mathcal{X}^*}, x^*(x) = 1\Big\} = \|A\|.$$
Now, by the above equality, we have
$$\|A\| \geq v(A) = v(\lambda A) \geq \sup\Big\{\mbox{Re}x^*(\lambda Ax)\,: \,\,x\in \mathbb{S}_\mathcal{X}, x^*\in\mathbb{S}_{\mathcal{X}^*}, x^*(x) = 1\Big\} = \|A\|.$$
Hence $\|A\| = v(A)$.

(ii)$\Rightarrow$(i) Let $\|A\| = v(A)$. For every $\varepsilon > 0$, we may find $x\in \mathbb{S}_\mathcal{X}$ and $x^*\in \mathbb{S}_{\mathcal{X}^*}$ such that
$x^*(x) = 1$ and $\big|x^*(Ax)\big| > \|A\| -\varepsilon$. Let $x^*(Ax) = \lambda\big|x^*(Ax)\big|$ for some $\lambda\in\mathbb{T}$. We have
\begin{align*}
1 + \|A\| \geq \|I + \overline{\lambda} A\| &\geq \|x + \overline{\lambda} Ax\|
\\& \geq \Big|x^*\big(x + \overline{\lambda} Ax\big)\Big| = \Big|x^*(x) + \overline{\lambda} x^*(Ax)\big)\Big|
\\& = \Big|1 + \overline{\lambda} \big(\lambda\big|x^*(Ax)\big|\big)\Big| = 1 + \big|x^*(Ax)\big|> 1 + \|A\| -\varepsilon.
\end{align*}
Hence $$1 + \|A\| \geq \|I + \overline{\lambda} A\|> 1 + \|A\| -\varepsilon.$$
Letting $\varepsilon\rightarrow0^+$, we obtain $\|I + \overline{\lambda} A\| = 1 + \|A\| $, or equivalently, $\|A + \lambda I\| = \|A\| + 1 $.
Thus $A\parallel I$.
\end{proof}

In the following proposition, we present a new proof of the previous result in Hilbert space context.
\begin{proposition}\label{normaloid}
Let $\mathcal{H}$ be a Hilbert space and $A\in \mathbb{B}(\mathcal{H})$. Then the following statements are equivalent:
\begin{itemize}
\item[(i)] $A\parallel I$.
\item[(ii)] $\|A\|=w(A)$.
\end{itemize}
\end{proposition}
\begin{proof}
(i)$\Rightarrow$(ii)
By \eqref{rel.N-P.BJ}, $I \perp _{BJ} \|A\|I - \lambda A$ for some $\lambda \in \mathbb{T}$. Using \cite{magajna} or \cite{B.S}, there exists a sequence $\{x_n\}$ of unit vectors such that
\begin{enumerate}
\item $\|x_n\|=1 \to \|I\|$ and
\item $\langle \overline{\lambda}A^*x_n, x_n\rangle \to \|A\|.$
\end{enumerate}
Then,
$$
\left|\: |\langle A^*x_n, x_n\rangle| - \|A\|\:\right|\leq \left| \langle \overline{\lambda} A^*x_n, x_n\rangle - \|A\| \right|\to 0
$$
when $n\to \infty$. Hence $w(A^*)=w(A)=\|A\|.$ From this we deduce that $A$ is normaloid.\\
(ii)$\Rightarrow$(i) Let (ii) holds. It is known that for any operator $A\in \mathbb{B}(\mathcal{H})$, $w(A)=\|A\|$ if and only if $r(A)=\|A\|$, where $r(A)$ is the spectral radius of $A$. So that there exists a sequence $\{y_n\}$ of unit vectors such that $|\langle A^*y_n, y_n\rangle| \to \|A\|.$ We denote $z_n=e^{i\theta_n}|\langle A^*y_n, y_n\rangle|$. By the complex Bolzano-Weierstrass Theorem, there exists a subsequence $\{\theta_{n_k}\}$ such that $e^{i\theta_{n_k}}\to e^{i\theta}$. Then $z_{n_k}\to e^{i\theta}\|A\|$ or equivalently
$\langle e^{i\theta}\|A\| I- A^*x_{n_k}, x_{n_k} \rangle\to 0$ from which we deduce that $I\perp _{BJ}e^{i\theta}\|A\|I- A$ and this completes the proof.
\end{proof}
\subsection{Norm--parallelism in locally uniformly convex spaces}
$\quad$

A Banach space $\mathcal{X}$ is said to be \textit{locally uniformly convex}
whenever for each $x\in \mathbb{S}_\mathcal{X}$ and each $0 <\varepsilon< 2$ there exists
some $0 < \delta<1$ such that $y\in \mathbb{B}_\mathcal{X}$, and
$\|x-y\|\geq\varepsilon$ imply $\left\|\frac{x+y}{2}\right\|<1-\delta$. Note that in a locally uniformly convex
the following condition holds:
if for any sequence $\{x_n\}$ in $\mathbb{B}_\mathcal{X}$ and for any $y$
in $\mathbb{S}_\mathcal{X}$ with $\lim\limits_{n\rightarrow\infty}\|x_n + y\| = 2$,
we have $\lim\limits_{n\rightarrow\infty}\|x_n - y\| = 0$.
It is obvious that every uniformly convex space is also locally uniformly convex.
\begin{lemma}\label{lemma-A-T}
Let $\mathcal{X}, \mathcal{Y}$ be Banach spaces. Suppose that $\mathcal{Y}$ is locally uniformly convex.
Let $A \in \mathbb{K}(\mathcal{X}, \mathcal{Y})$ and $B \in \mathbb{B}(\mathcal{X}, \mathcal{Y})$. Suppose that $A \neq 0 \neq B$.
If $A\parallel B$, then there are $\lambda \in \mathbb{T}$, $\{x_n\}\subset \mathbb{S}_\mathcal{X}$ and
$y \in \mathbb{S}_\mathcal{Y}$ such that
$$\lim\limits_{n\rightarrow\infty}\frac{A}{\|A\|}x_n = y, \quad\qquad \lim\limits_{n\rightarrow\infty}\lambda\frac{B}{\|B\|}x_n = y.$$
\end{lemma}
\begin{proof}
Suppose that $A\parallel B$ holds.
Since the norm--parallelism is $\mathbb{R}$-homogenous, we have $\frac{A}{\|A\|}\parallel \frac{B}{\|B\|}$.
Hence, there exists $\lambda\in\mathbb{T}$ such that
$$\left\|\frac{A}{\|A\|} + \lambda \frac{B}{\|B\|}\right\| = \left\|\frac{A}{\|A\|}\right\| + \left\|\frac{B}{\|B\|}\right\| = 2.$$
Since
$$\sup\left\{\left\|\frac{A}{\|A\|}x_n + \lambda \frac{B}{\|B\|}x_n\right\|:\, x_n\in \mathbb{S}_\mathcal{X}\right\} = \left\|\frac{A}{\|A\|} + \lambda \frac{B}{\|B\|}\right\|,$$
there exists a sequence of unit vectors $\{x_n\}$ in $\mathcal{X}$ such that
\begin{align}\label{id.2.0-A-T}
\lim_{n\rightarrow\infty} \left\|\frac{A}{\|A\|}x_n + \lambda \frac{B}{\|B\|}x_n\right\| = 2.
\end{align}
By virtue of compactness of $A$, there exist a subsequence $\{x_{n_k}\}$ and $y \in \mathbb{S}_\mathcal{Y}$ such that
\begin{align}\label{id.2.1-A-T}
\lim_{k\rightarrow\infty} \frac{A}{\|A\|}x_{n_k} = y.
\end{align}
From
\begin{align*}
\left\|\frac{A}{\|A\|}x_{n_k} + \lambda \frac{B}{\|B\|}x_{n_k}\right\|& \leq\left\|\frac{A}{\|A\|}x_{n_k} -y\right\|+\left\|y + \lambda \frac{B}{\|B\|}x_{n_k}\right\|
\\&\leq\left\|\frac{A}{\|A\|}x_{n_k} -y\right\|+\|y\| + \left\|\lambda \frac{B}{\|B\|}x_{n_k}\right\|
\\&\leq\left\|\frac{A}{\|A\|}x_{n_k} -y\right\|+1 + 1
\end{align*}
as well as (\ref{id.2.0-A-T}) and (\ref{id.2.1-A-T}) we obtain
$\lim\limits_{n\rightarrow\infty} \left\|\lambda \frac{B}{\|B\|}x_{n_k}\right\| = 1$ and
$\lim\limits_{n\rightarrow\infty} \left\|y + \lambda \frac{B}{\|B\|}x_{n_k}\right\| = 2$.
Since $\mathcal{Y}$ is locally uniformly convex, we infer that
$\lim_{n\rightarrow\infty} \left\|\lambda\frac{B}{\|B\|}x_{n_k} - y\right\| = 0$.
Thus $\lim_{n\rightarrow\infty} \lambda\frac{B}{\|B\|}x_{n_k} = y$.
Now, the proof is completed.
\end{proof}
\begin{theorem}\label{th.1}
Let $\mathcal{X}$ be a closed subspace of a locally uniformly convex Banach space $\mathcal{Y}$.
Let $J\in \mathbb{B}(\mathcal{X}, \mathcal{Y})$ denote the inclusion operator (i.e., $Jx = x$ for all $x$ in $\mathcal{X}$).
Let $A\in\mathbb{K}(\mathcal{X}, \mathcal{Y})$.
Then the following statements are equivalent:
\begin{itemize}
\item[(i)] $A\parallel J$.
\item[(ii)] $\lambda\|A\|$ is an eigenvalue of $A$ for some $\lambda\in\mathbb{T}$.
\end{itemize}
\end{theorem}
\begin{proof}
(i)$\Rightarrow$(ii) Suppose that $A\parallel J$ holds.
It follows from Lemma \ref{lemma-A-T} that,
there exist a subsequence $\{x_n\}\subset \mathbb{S}_\mathcal{X}$ and $y\in \mathbb{S}_\mathcal{X}$ and $\lambda\in \mathbb{T}$ such that
\begin{align}\label{id.2.2}
\lim_{k\rightarrow\infty} \frac{A}{\|A\|}x_n = y, \quad\qquad\lim_{k\rightarrow\infty} \lambda x_n = y.
\end{align}
Due to $A$ is continuous and using the second equality in (\ref{id.2.2}), we get
\begin{align}\label{id.2.3}
\lim_{k\rightarrow\infty} \frac{A}{\|A\|}(\lambda x_n) = \frac{A}{\|A\|}y.
\end{align}
By the first equality in (\ref{id.2.2}) and by (\ref{id.2.3}) we reach that $\frac{A}{\|A\|}y = \lambda y$, or equivalently, $Ay = \lambda\|A\|y$.
Thus $\lambda \|A\|$ is an eigenvalue of $A$.

(ii)$\Rightarrow$(i) Suppose (ii) holds. So, there exists $x\in \mathcal{X}\setminus\{0\}$ such that $Ax = \lambda\|A\|x$. We have
\begin{align*}
\|A + \lambda J\| \geq \left\|(A + \lambda J)\frac{x}{\|x\|}\right\| &= \left\|\lambda\|A\|\frac{x}{\|x\|} + \lambda\frac{x}{\|x\|}\right\| = \|A\| + 1 \geq \|A + \lambda J\|.
\end{align*}
Thus $\|A + \lambda J\| = \|A\|+1=\|A\| + \|J\|$, which means that $A\parallel J$.
\end{proof}
\begin{remark} \label{re.1.1}
Notice that the condition of compactness in the implication (i)$\Rightarrow $(ii) of Theorem \ref{th.1} is essential.
For example, consider the right shift operator $A\,:\ell^2\longrightarrow \ell^2$ defined by
$A(\xi_1, \xi_2, \xi_3, \cdots) = (0, \xi_1, \xi_2, \xi_3, \cdots)$.
It is easily seen that $A\parallel I$ but $A$ has no eigenvalues.
\end{remark}
Let $(\Omega, \mathfrak{M}, \rho)$ be a measure space. It is well known that every $\mathbf{L}^p(\Omega, \mathfrak{M}, \rho)$ space with $1<p<\infty$ is a uniformly convex Banach space. Therefore, as a consequence of Theorem \ref{th.1}, we have the following result.
\begin{corollary}\label{cr.2}
Let $(\Omega, \mathfrak{M}, \rho)$ be a measure space and $A\in\mathbb{K}\big(\mathbf{L}^p((\Omega, \mathfrak{M}, \rho)\big)$ with $1<p<\infty$. Then the following statements are equivalent:
\begin{itemize}
\item[(i)] $A\parallel I$.
\item[(ii)] $\lambda\|A\|$ is an eigenvalue of $A$ for some $\lambda\in\mathbb{T}$.
\end{itemize}
\end{corollary}
\begin{remark} \label{re.2.1}
Notice that in the implication (i)$\Rightarrow $(ii) of Corollary \ref{cr.2} the condition $1<p<\infty$ is essential.
For example, let $x(t) = \sin(\pi t)$ and $y(t) = \cos(\pi t)$ with $0\leq t\leq 1$. Consider the rank-one
operator $A = x\otimes y$. Then $A\in\mathbb{K}\big(\mathbf{L}^1([0, 1])\big)\cup\mathbb{K}\big(\mathbf{L}^{\infty}([0, 1])\big)$.
It is easily seen that $A\parallel I$ but the operator $A$ has no non-zero eigenvalue.
\end{remark}
Let $\mathcal{X}$ be a normed space and $A\in\mathbb{B}(\mathcal{X})$. Recall that an invariant subspace for $A$
is a closed linear subspace $\mathcal{X}_0$ of $\mathcal{X}$ such that $A(\mathcal{X}_0)\subseteq \mathcal{X}_0$. The following result shows that the notion
of parallelism is related to the invariant subspace problem.
\begin{corollary}
Let $\mathcal{X}$ be a locally uniformly convex Banach space and $A\in\mathbb{K}(\mathcal{X})$.
If $A\parallel I$, then the operator $A$ has a invariant subspace of dimension one.
\end{corollary}
\begin{proof}
Let $A\parallel I$. It follows from Theorem \ref{th.1} that there exist $\lambda\in\mathbb{T}$ and $x_0\in \mathcal{X}\setminus\{0\}$ such that $Ax_0 = \lambda\|A\|x_0$.
Set $\mathcal{X}_0: = \text{span}\{Ax_0\}$. It is easy to verify that $A(\mathcal{X}_0)\subseteq \mathcal{X}_0$ and $\dim \mathcal{X}_0 = 1$.
\end{proof}
\subsection{Norm--parallelism of nilpotent and projections}
$\quad$

In this section we investigate nilpotent and projections in the context of norm--parallelism in locally uniformly convex spaces.
First, we show that iterations of nilpotent are not norm--parallel.
\begin{theorem}\label{th-nilp}
Let $\mathcal{X}$ be a locally uniformly convex Banach space.
Let $A\in \mathbb{K}(\mathcal{X})$ such that $A^{m-1}\neq 0$, $A^m=0$ for some $m\in \mathbb{N}$.
Then $A^k\nparallel A^{j}$ for every $1\leq k< j\leq m$.
\end{theorem}
\begin{proof}
Assume, contrary to
our claim, that $A^k\parallel A^{j}$ for some $1\leq k< j\leq m$.
Since $A$ is compact, $A^k, A^{j}$ are compact.
It follows from Lemma \ref{lemma-A-T} that,
there exist a sequence $\{x_{n}\}\subset \mathbb{S}_\mathcal{X}$, $y\in \mathbb{S}_\mathcal{X}$ and $\lambda\in \mathbb{T}$ such that
\begin{align}\label{id.2.2-n}
\lim_{n\rightarrow\infty} \frac{A^{k}}{\|A^{k}\|}x_n = y,\quad\qquad \lim_{n\rightarrow\infty} \lambda\frac{A^{j}}{\|A^{j}\|}x_n = y.
\end{align}
Due to $A^{j-k}$ is continuous, we get by (\ref{id.2.2-n})
$\lim_{n\rightarrow\infty} A^{j-k}\left(\frac{A^{k}}{\|A^{k}\|}x_n\right) = A^{j-k}y$. Hence
\begin{align}\label{id.2.3-n}
\lim_{n\rightarrow\infty}\left(\frac{A^{j}}{\|A^{k}\|}x_n\right) = A^{j-k}y.
\end{align}
Now the equality (\ref{id.2.3-n}) becomes
\begin{align}\label{nilp-ada}
\lim_{n\rightarrow\infty}\left(\frac{A^{j}}{\|A^{j}\|}x_n\right) = \frac{\|A^k\|}{\|A^j\|}A^{j-k}y.
\end{align}
By (\ref{id.2.2-n}) and (\ref{nilp-ada}) we reach that $\frac{\|A^k\|}{\|A^j\|}A^{j-k}y = \overline{\lambda}y$.
We obtain $A^{j-k}y=\alpha y$ with $\alpha:= \overline{\lambda}\frac{\|A^j\|}{\|A^k\|}$.
Therefore, $A^{m(j-k)}y=\alpha^m y\neq 0$, while $A^{m(j-k)}=0$, and we obtain a contradiction.
\end{proof}
Now, we investigate whether projections may be norm--parallel.
\begin{theorem}\label{th-proj}
Let $\mathcal{X}$ be a locally uniformly convex Banach space.
Let $A, B\in \mathbb{B}(\mathcal{X})$ be operators such that $A^2=A$ and $B^2=B$. Moreover, suppose that $\dim A(\mathcal{X})<\infty$.
If $A\parallel B$, then $A(\mathcal{X})\cap B(\mathcal{X})$ is a nontrivial subspace.
\end{theorem}
\begin{proof}
Since $\dim A(\mathcal{X})<\infty$, $A\in \mathbb{K}(\mathcal{X})$.
It follows from Lemma \ref{lemma-A-T} that,
there exist a sequence $\lambda\in \mathbb{T}$, $\{x_{n}\}\subset \mathbb{S}_\mathcal{X}$ and $y\in \mathbb{S}_\mathcal{X}$ such that
\begin{align}\label{id.2.2-p}
\lim_{k\rightarrow\infty} \frac{A}{\|A\|}x_n = y, \qquad\quad \lim_{k\rightarrow\infty} \lambda\frac{B}{\|B\|}x_n = y.
\end{align}
The operators $A, B$ are continuous. Thus we get by (\ref{id.2.2-p})
\begin{align}\label{id.2.3-p}
\lim_{k\rightarrow\infty} \frac{A^2}{\|A\|}x_n = Ay, \qquad\quad \lim_{k\rightarrow\infty} \lambda\frac{B^2}{\|B\|}x_n = By.
\end{align}
Since $A^2=A$, $B^2=B$, it follows from (\ref{id.2.3-p}) that
\begin{align}\label{id.2.4-p}
\lim_{k\rightarrow\infty} \frac{A}{\|A\|}x_n = Ay, \qquad\quad \lim_{k\rightarrow\infty} \lambda\frac{B}{\|B\|}x_n = By.
\end{align}
Combining (\ref{id.2.2-p}) and (\ref{id.2.4-p})
we get $Ay= y = By$. So, it yields $y\in A(\mathcal{X})\cap B(\mathcal{X})$,
which means that $\mbox{span}\{y\}\subset A(\mathcal{X})\cap B(\mathcal{X})$.
\end{proof}
\begin{corollary}\label{cr-proj-2}
Let $\mathcal{X}$ be a locally uniformly convex Banach space.
Let $A, B\in \mathbb{B}(\mathcal{X})$ be projections such that $\|A\|=1$ and $\|B\|=1$. Moreover, suppose that $\dim A(\mathcal{X})<\infty$.
Then the following conditions are equivalent:
\begin{itemize}
\item[(i)] $A\parallel B$.
\item[(ii)] $A(\mathcal{X})\cap B(\mathcal{X})$ is a nontrivial subspace.
\end{itemize}
\end{corollary}
\begin{proof}
The implication (i)$\Rightarrow$(ii) holds by Theorem \ref{th-proj}.
We prove (ii)$\Rightarrow$(i). Fix $x\in A(\mathcal{X})\cap B(\mathcal{X})\cap \mathbb{S}_\mathcal{X}$. It follows that $Ax=x$ and $Bx=x$. Thus
$$\|A\|+\|B\|= 2 = \|x + x\| = \|Ax + Bx\|\leq\|A + B\|\leq\|A\| + \|B\|,$$
which completes the proof of this theorem.
\end{proof}
\subsection{Norm--parallelism of compact operators}
$\quad$

In this section, we give some equivalence assertions about the norm--parallelism
of compact operators.
Let $0\leq \varepsilon <1$. We say that a mapping $U\,:\mathcal{X}\longrightarrow \mathcal{Y}$ is an $\varepsilon$-isometry if
$$(1 - \varepsilon)\|x\|\leq \|Ux\| \leq(1 + \varepsilon)\|x\|\qquad (x\in \mathcal{X}).$$
\begin{theorem}
Let $\mathcal{X}$ be a normed space and let $A, B\in \mathbb{B}(\mathcal{X})$. Suppose that for every $\varepsilon > 0$ there exist a
normed space $\mathcal{Y}$ and a surjective $\varepsilon$-isometry $U\,:\mathcal{X}\longrightarrow \mathcal{Y}$ such that
$UAU^{-1}\parallel UBU^{-1}$ in $\mathbb{B}(\mathcal{Y})$. Then $A\parallel B$.
\end{theorem}
\begin{proof}
Fix $\varepsilon > 0$. By assumption, there exist a
normed space $\mathcal{Y}$ and a surjective $\varepsilon$-isometry $U\,:\mathcal{X}\longrightarrow \mathcal{Y}$ such that
$UAU^{-1}\parallel UBU^{-1}$. Hence there exists $\lambda\in\mathbb{T}$ such that
\begin{align}\label{id.8.3001}
\big\|UAU^{-1} + \lambda UBU^{-1}\big\| = \|UAU^{-1}\| + \|UBU^{-1}\|.
\end{align}
For every $C\in \mathbb{B}(\mathcal{X})$ we have
\begin{align*}
\|UCU^{-1}y\| \leq \|UC\|\,\|U^{-1}y\|\leq \|U\|\,\|C\|\,\frac{\|y\|}{1 - \varepsilon} \leq \frac{1 + \varepsilon}{1 - \varepsilon}\|C\|\,\|y\| \quad (y\in \mathcal{Y})
\end{align*}
which implies that
\begin{align}\label{id.8.3}
\|UCU^{-1}\| \leq \frac{1 + \varepsilon}{1 - \varepsilon}\|C\|.
\end{align}
On the other hand, we have
\begin{align*}
\frac{1 - \varepsilon}{1 + \varepsilon}\|Cx\| &= \frac{1 - \varepsilon}{1 + \varepsilon}\|U^{-1}\big(UCU^{-1}\big)Ux\|
\\& \leq \frac{1 - \varepsilon}{1 + \varepsilon}\|U^{-1}\big(UCU^{-1}\big)\|\,\|Ux\|
\\&\leq \frac{1 - \varepsilon}{1 + \varepsilon}\|U^{-1}\|\,\|UCU^{-1}\|(1 + \varepsilon)\|x\|
\\& \leq \frac{1 - \varepsilon}{1 + \varepsilon}\times\frac{1}{1 - \varepsilon}\|UCU^{-1}\|(1 + \varepsilon)\|x\| = \|UCU^{-1}\|\,\|x\| \quad (x\in \mathcal{X}),
\end{align*}
whence
\begin{align}\label{id.8.4}
\frac{1 - \varepsilon}{1 + \varepsilon}\|C\| \leq \|UCU^{-1}\|.
\end{align}
From (\ref{id.8.3}) and (\ref{id.8.4}) we therefore get
\begin{align}\label{id.8.5}
\frac{1 - \varepsilon}{1 + \varepsilon}\|C\| \leq \|UCU^{-1}\|\leq \frac{1 + \varepsilon}{1 - \varepsilon}\|C\|.
\end{align}
We have
\begin{align*}
\|A\| + \|B\|&\geq \|A + \lambda B\|
\\& \geq \frac{1 - \varepsilon}{1 + \varepsilon}\|U(A + \lambda B)U^{-1}\| \hspace{3cm}\big(\mbox{by}\,(\ref{id.8.5})\, \mbox{for}\, C = A + \lambda B \big)
\\& = \frac{1 - \varepsilon}{1 + \varepsilon}\big(\|UAU^{-1}\| + \|UBU^{-1}\|\big) \hspace{4.7cm}\big(\mbox{by}\,(\ref{id.8.3001})\big)
\\& \geq \frac{1 - \varepsilon}{1 + \varepsilon}\left(\frac{1 - \varepsilon}{1 + \varepsilon}\|A\| + \frac{1 - \varepsilon}{1 + \varepsilon}\|B\|\right) \hspace{2cm}\big(\mbox{by}\,(\ref{id.8.5}) \big)
\\& = \left(\frac{1 - \varepsilon}{1 + \varepsilon}\right)^2(\|A\| + \|B\|).
\end{align*}
Thus
\begin{align}\label{id.8.5006}
\|A\| + \|B\|\geq \|A + \lambda B\| \geq \left(\frac{1 - \varepsilon}{1 + \varepsilon}\right)^2(\|A\| + \|B\|).
\end{align}
Letting $\varepsilon \rightarrow 0^{+}$ in (\ref{id.8.5006}), we obtain $\|A + \lambda B\| = \|A\| + \|B\|$, hence $A\parallel B$.
\end{proof}
\begin{proposition}\label{pr 3.147}
Let $\mathcal{X}, \mathcal{Y}$ be normed space, $A\in\mathbb{B}(\mathcal{X}, \mathcal{Y})$ and $x, y\in \mathbb{M}_A$. If $Ax\parallel Ay$, then
$x\parallel y$.
\end{proposition}
\begin{proof}
Let $Ax\parallel Ay$. Hence there exists $\lambda_0\in\mathbb{T}$ such that
$\|Ax + \lambda_0 Ay\| = \|Ax\| + \|Ay\|$. Since $x, y\in \mathbb{M}_A$, we obtain $\|Ax + \lambda_0 Ay\| = 2\|A\|$.
Now, let $x\nparallel y$. Then $0< \|x + \lambda y\|< \|x\| + \|y\| = 2$ for all $\lambda\in\mathbb{T}$. In particular, $0< \|x + \lambda_0 y\|<2$.
So, we have
\begin{align*}
\|A\| &\geq \left\|A\left(\frac{x + \lambda_0 y}{\|x + \lambda_0 y\|}\right)\right\|
\\& = \frac{1}{\|x + \lambda_0 y\|}\|Ax + \lambda_0 Ay\|
\\& = \frac{2\|A\|}{\|x + \lambda_0 y\|} > \frac{2\|A\|}{2} = \|A\|,
\end{align*}
which is a contradiction. Thus $x\parallel y$.
\end{proof}
In the sequel, we show that the converse of Proposition \ref{pr 3.147} is also true if both $\mathcal{X}, \mathcal{Y}$ are real smooth
Banach spaces. To this end, let us quote a result from \cite{D. Sai}.
\begin{lemma}\cite[Theorem 3.1]{D. Sai}\label{lem 684}
Let $\mathcal{X}$ be a real smooth Banach space, $A\in\mathbb{B}(\mathcal{X})$ and $x\in \mathbb{M}_A$. Then
$$A\Big(\big\{z\in \mathcal{X}; \quad x\perp_{BJ} z\big\}\Big) \subseteq \big\{w\in \mathcal{X}; \quad Ax\perp_{BJ} w\big\}.$$
\end{lemma}
We are now in a position to establish one of our main results.
\begin{theorem}\label{th 39612}
Let $\mathcal{X}$ be a real smooth Banach space, $A\in\mathbb{B}(\mathcal{X})$ with $\|A\| = 1$ and $x, y\in \mathbb{M}_A$.
Then the following statements are equivalent:
\begin{itemize}
\item[(i)] $x\parallel y$.
\item[(ii)] $Ax\parallel Ay$.
\end{itemize}
\end{theorem}
\begin{proof}
By Proposition \ref{pr 3.147}, (ii) implies (i).

(i)$\Rightarrow$(ii)
Let $x\parallel y$. By (\ref{rel.N-P.BJ}), there exists $\lambda\in\mathbb{T}$ such that $x\perp_{BJ}\big(\|y\|x + \lambda \|x\|y\big)$.
From $x, y\in \mathbb{S}_\mathcal{X}$ we deduce that $x\perp_{BJ}\big(x + \lambda y\big)$. Thus $$A(x + \lambda y) \in A(\{z\in \mathcal{X}: x\perp_{BJ} z\}).$$ From Lemma \ref{lem 684} we therefore conclude that $Ax\perp_{BJ} A(x + \lambda y)$. Thus $Ax\perp_{BJ} \|Ay\|Ax + \lambda \|Ax\|Ay$, since $\|Ax\| = \|Ay\| = 1$.
Again applying (\ref{rel.N-P.BJ}), we obtain $Ax\parallel Ay$.
\end{proof}
\begin{remark}
Notice that the smoothness condition of Banach space in the above theorem is essential. For example,
let us consider the space $\mathbb{R}^2$ with the max--norm.
Consider the norm one operator $A\in \mathbb{B}(\mathbb{R}^2)$ defined by $A(x, y) = \frac{1}{2}(x + y, x - y)$.
It is easily seen that $(1, -1), (-1, -1)\in \mathbb{M}_A$ and
$(1, -1)\parallel (-1, -1)$. However, we have $A(1, -1) = (0, 1)\nparallel (-1, 0)=A(-1, -1)$.
\end{remark}
Because of any compact operator on a reflexive Banach space must attain its norm, we have the following result as a consequence of Theorem \ref{th 39612}.
\begin{corollary}
Let $\mathcal{X}$ be a real reflexive smooth Banach space and $A\in\mathbb{K}(\mathcal{X})$ with $\|A\| = 1$. Then there exists
a unit vector $x\in \mathcal{X}$ such that for any $y\in \mathbb{M}_A$, the following statements are equivalent:
\begin{itemize}
\item[(i)] $x\parallel y$.
\item[(ii)] $Ax\parallel Ay$.
\end{itemize}
\end{corollary}
Next, let $A, B\in \mathbb{B}(\mathcal{X}, \mathcal{Y})$. If $x\in \mathbb{M}_A\cap \mathbb{M}_B$ and $Ax\parallel Bx$, then there exists $\lambda\in\mathbb{T}$ such that
$\|Ax + \lambda Bx\| = \|Ax\| + \|Bx\|.$
Hence
$$\|A\| + \|B\| = \|Ax\| + \|Bx\| = \|Ax + \lambda Bx\| \leq \|A + \lambda B\| \leq \|A\| + \|B\|.$$
Thus $\|A + \lambda B\| = \|A\| + \|B\|$, and hence $A\parallel B$.
There are examples in which $A\parallel B$ but not $Ax\parallel Bx$ for any
$x\in \mathbb{M}_A\cap \mathbb{M}_B$ (see \cite[Example 2.17]{Z.M.2}).

In a Hilbert space $\mathcal{H}$ and for $A, B\in\mathbb{B}(\mathcal{H})$, we \cite[Crolloraly 4.2]{Z.M.1} proved that
$A\parallel B$ if and only if there exists a sequence of unit vectors $\{\xi_n\}$ in $\mathcal{H}$ such that
\begin{align*}
\lim_{n\rightarrow\infty} \big|\langle A\xi_n, B\xi_n\rangle\big| = \|A\|\,\|B\|.
\end{align*}
It follows that if the Hilbert space $\mathcal{H}$ is finite dimensional, then $A\parallel B$ if and only if $Ax\parallel Bx$ for some $x\in \mathbb{M}_A\cap \mathbb{M}_B$.

In the case when $\mathcal{H}$ is infinite dimensional, there are examples showing that $A\parallel B$ but there is no $x\in \mathbb{M}_A\cap \mathbb{M}_B$ such that $Ax\parallel Bx$ (see \cite[Example 2.17]{Z.M.2}). This indicates that for such a result to be true in an infinite dimensional Hilbert space, we need to impose certain additional condition(s).
In \cite[Theorem 2.18]{Z.M.2} it is proved that for $A\in\mathbb{B}(\mathcal{H})$, if $\mathbb{S}_{\mathcal{H}_0} = \mathbb{M}_A$ where $\mathcal{H}_0$ is a finite dimensional
subspace of $\mathcal{H}$ and $\sup\{\|Az\| :\, z\in{\mathcal{H}_0}^\perp,\, \|z\| = 1\}< \|A\|$, then for any $B\in\mathbb{B}(\mathcal{H})$, $A\parallel B$
if and only if there exists a vector $x\in \mathbb{M}_A\cap \mathbb{M}_B$ such that $Ax\parallel Bx$.

Furthermore, for $A, B\in\mathbb{K}(\mathcal{H})$ it is proved in \cite[Theorem 2.10]{Z} that
$A\parallel B \Leftrightarrow Ax\parallel Bx$ for some $x\in \mathbb{M}_A\cap \mathbb{M}_B$.
Notice that the condition of compactness is essential (see \cite[Example 2.7]{Z}).

The following auxiliary result is needed in our next theorem.
\begin{lemma}\cite[Theorem 3.1]{w}\label{lemma.2.0199}
Let $\mathcal{X}, \mathcal{Y}$ be real reflexive Banach spaces, $\mathcal{Y}$ be smooth and strictly convex and $A, B\in \mathbb{K}(\mathcal{X}, \mathcal{Y})$. Let either $\mathbb{M}_A$ be connected or $\mathbb{M}_A = \{-u, +u\}$ for some unit vector $u\in \mathcal{X}$.
If $A\perp_{BJ} B$, then there exists a vector $x\in \mathbb{M}_A$ such that $Ax\perp_{BJ} Bx$.
\end{lemma}
\begin{theorem}\label{th.17.5}
Let $\mathcal{X}, \mathcal{Y}$ be real reflexive Banach spaces, $\mathcal{Y}$ be smooth and strictly convex and $A, B\in \mathbb{K}(\mathcal{X}, \mathcal{Y})$. Let either $\mathbb{M}_A$ be connected or $\mathbb{M}_A = \{-u, +u\}$ for some unit vector $u\in \mathcal{X}$. Then the following statements are equivalent:
\begin{itemize}
\item[(i)] $A\parallel B$.
\item[(ii)] There exists a vector $x\in \mathbb{M}_A\cap \mathbb{M}_B$ such that $Ax\parallel Bx$.
\end{itemize}
In addition, if $x$ satisfying (ii), then $\frac{Ax}{\|A\|} = \pm\frac{Bx}{\|B\|}$.
\end{theorem}
\begin{proof}
Let $A\parallel B$. By (\ref{rel.N-P.BJ}), there exists $\lambda\in\mathbb{T}$ such that $A\perp_{BJ}\big(\|B\|A + \lambda \|A\|B\big)$.
Noting that since $\mathcal{X}$ is a reflexive Banach space and $A\in \mathbb{K}(\mathcal{X}, \mathcal{Y})$ we have $\mathbb{M}_A \neq\emptyset$.
It follows from Lemma \ref{lemma.2.0199} that there exists a vector $x\in \mathbb{M}_A$ such that $Ax\perp_{BJ} \big(\|B\|Ax + \lambda \|A\|Bx\big)$.
Hence, $\Big\|Ax + \mu \big(\|B\|Ax + \lambda \|A\|Bx\big)\Big\|\geq \|Ax\|$ for all $\mu \in \mathbb{R}$. Let $\mu = - \frac{1}{\|B\|}$. Then
$$\Big\|Ax - \frac{1}{\|B\|}\big(\|B\|Ax + \lambda \|A\|Bx\big)\Big\|\geq \|Ax\|.$$
Thus $\|A\|\|Bx\|\geq \|B\|\|Ax\|$. Since $\|Ax\| = \|A\|$, we get
$\|Bx\|\geq \|B\|$. So $\|B\|=\|Bx\|$ and hence $x\in \mathbb{M}_B$. Since $Ax\perp_{BJ} \big(\|B\|Ax + \lambda \|A\|Bx\big)$, we get
$Ax\perp_{BJ} \big(\|Bx\|Ax + \lambda \|Ax\|Bx\big)$. Again applying (\ref{rel.N-P.BJ}), we obtain $Ax\parallel Bx$.
Then $\|Ax + \alpha Bx\| = \|Ax\| + \|Bx\|$ for some $\alpha \in\mathbb{T}$, which by the strict convexity of $\mathcal{Y}$ we get $Ax = \pm cBx$ with $c>0$. Consequently $c = \frac{\|Ax\|}{\|Bx\|} = \frac{\|A\|}{\|B\|}$. Thus $\frac{Ax}{\|A\|} =\frac{\pm cBx}{c\|B\|} = \pm \frac{Bx}{\|B\|}$.

The converse is obvious.
\end{proof}
\begin{remark} \label{re.1.1888}
If $\mathcal{H}$ is a real finite-dimensional Hilbert space, then $\mathbb{B}(\mathcal{H}) = \mathbb{K}(\mathcal{H})$.
It is well known that $\mathbb{M}_A \neq \emptyset$ for every $A\in \mathbb{K}(\mathcal{H})$.
Also, it is easy to see that either $\mathbb{M}_A$ is connected or $\mathbb{M}_A = \{-u, +u\}$ for some unit vector $u\in \mathcal{X}$. So,
as an immediate consequence of Theorem \ref{th.17.5}, we get Theorem 2.13 of \cite{Z.M.2}.
\end{remark}
\begin{corollary}\label{cr.20}
Let $\mathcal{X}, \mathcal{Y}$ be real reflexive Banach spaces. Let $\mathcal{Y}$ be smooth and strictly convex and there exist $[\cdot, \cdot]:\mathcal{Y}\times \mathcal{Y}\rightarrow\mathbb{R}$ a semi
inner product generating its norm. Let $A, B\in \mathbb{K}(\mathcal{X}, \mathcal{Y})$ and $\mathbb{M}_A$ be either connected or $\mathbb{M}_A = \{-u, +u\}$ for some unit vector $u\in \mathcal{X}$. If $A\parallel B$, then
$$\|B\| = \sup\big\{|[Bx, y]|: \,\, x\in \mathbb{S}_\mathcal{X}, y\in \mathbb{S}_\mathcal{Y}, Ax\parallel y\big\}.$$
\end{corollary}
\begin{proof}
Obviously, we have
$$\sup\big\{|[Bx, y]|: \,\, x\in \mathbb{S}_\mathcal{X}, y\in \mathbb{S}_\mathcal{Y}, Ax\parallel y\big\} \leq \|B\|.$$
Due to $A\parallel B$, by Theorem \ref{th.17.5}, there exists a vector $x_0\in \mathbb{M}_A\cap \mathbb{M}_B$ such that $\frac{Ax_0}{\|A\|} = \pm \frac{Bx_0}{\|B\|}$.
Put $x := x_0$ and $y := \frac{Ax_0}{\|Ax_0\|}$. Then $x\in \mathbb{S}_\mathcal{X}, y\in \mathbb{S}_\mathcal{Y}$ and $Ax\parallel y$. We have
$$\big|[Bx, y]\big| = \Big|\Big[Bx_0, \frac{Ax_0}{\|Ax_0\|}\Big]\Big| = \Big|\Big[Bx_0, \pm \frac{Bx_0}{\|B\|}\Big]\Big| = \frac{\|Bx_0\|^2}{\|B\|} = \|B\|.$$
Thus the supremum is attained.
\end{proof}
Since every finite dimensional normed space is reflexive on which
every linear operator is compact, as a consequence of Theorem \ref{th.17.5}, we have the following result.
\begin{corollary}\label{cr.2012}
Let $\mathcal{X}$ be a finite dimensional real normed space and $\mathcal{Y}$ be any real normed space. Assume $A, B\in \mathbb{B}(\mathcal{X}, \mathcal{Y})$ and $\mathbb{M}_A = \mathbb{S}_\mathcal{X}$.
Then the following statements are equivalent:
\begin{itemize}
\item[(i)] $A\parallel B$.
\item[(ii)] There exists $x\in \mathbb{M}_B$ such that $Ax\parallel Bx$.
\end{itemize}
\end{corollary}

\textbf{Acknowledgement.} M. S. Moslehian (the corresponding author) was supported by a grant from Ferdowsi University of Mashhad (No. 1/43523).

\bibliographystyle{amsplain}

\end{document}